\documentclass[11pt,a4paper]{article}
\usepackage{bbm}
\usepackage{graphicx}
\graphicspath{ {./images/} }
\usepackage{wrapfig}
\usepackage[mathscr]{euscript}
\usepackage{subcaption}
\usepackage{hyperref}
\usepackage{epstopdf}
\usepackage{enumerate}
\usepackage{amsmath,amsfonts,amssymb,amsthm,epsfig,epstopdf,titling,url,array}
\usepackage{color}
\usepackage{framed}
\usepackage[utf8]{inputenc}
\usepackage[english]{babel}

\usepackage{xparse}

\NewDocumentCommand{\INTERVALINNARDS}{ m m }{
	#1 {,} #2
}
\NewDocumentCommand{\interval}{ s m >{\SplitArgument{1}{,}}m m o }
{
	\IfBooleanTF{#1}{
		\left#2 \INTERVALINNARDS #3 \right#4
	}{
		\IfValueTF{#5}{
			#5{#2} \INTERVALINNARDS #3 #5{#4}
		}{
			#2 \INTERVALINNARDS #3 #4
		}
	}
}
\newtheorem{theorem}{Theorem}[section]
\newtheorem{corollary}{Corollary}[theorem]
\newtheorem{definition}{Definition}[section]
\newtheorem{lemma}[theorem]{Lemma}
\newtheorem{example}{Example}[section]

\newtheorem{Conjecture}{Conjecture}
\newtheorem{Observation}{Observation}[section]
\newtheorem{Remark}{Remark}[section]

\usepackage{authblk}

\begin{document} 
	\title{\textbf{ Symmetry and dynamics of Chebyshev's method}}
	\author[1]{Tarakanta Nayak\footnote{tnayak@iitbbs.ac.in}}
	\author[1]{Soumen Pal\footnote{Corresponding author, sp58@iitbbs.ac.in}}
	\affil[1]{School of Basic Sciences, 
		Indian Institute of Technology Bhubaneswar, India}
	\date{}
	\maketitle
	\begin{abstract}
    The set of all holomorphic Euclidean isometries  preserving the Julia set of a rational map $R$ is denoted by $\Sigma R$. It is shown in this article that if a root-finding method $F$ satisfies the Scaling theorem, i.e., for a polynomial $p$, $F_p$ is affine conjugate to $F_{\lambda p \circ T}$ for every nonzero complex number $\lambda $ and every affine map $T$,  then for a centered polynomial $p$ of order at least two (which is not a monomial), $\Sigma p\subseteq \Sigma F_p$. As the Chebyshev's method  satisfies the Scaling theorem, we have  $\Sigma p \subseteq \Sigma {C_p}$, where $p$ is a centered polynomial. The rest part of this article is devoted to explore the situations where the equality holds and in the process, the dynamics of $C_p$ is found.  We show that  the Julia set $\mathcal{J}(C_p)$ of $ C_p$ can never be a line. If a centered polynomial $p$ is (a) unicritical, (b) having exactly two roots with the  same multiplicity, (c) cubic and $\Sigma p$ is non-trivial or (d) quartic, $0$ is a root of $p$ and $\Sigma p $ is non-trivial then it is proved that $\Sigma p = \Sigma C_p$. It is found in all these cases that  the Fatou set $\mathcal{F}(C_p)$ is the union of all the attracting basins of $C_p$ corresponding to the roots of $p$ and  $\mathcal{J}(C_p)$ is connected. It is observed that $\mathcal{J}(C_p)$ is locally connected in all these cases.
	\end{abstract}
\section{Introduction}
	Finding the roots of a polynomial is a  classical problem in mathematical sciences. A root-finding method $F$ is a function that associates  a given polynomial $p$ with a rational map $F_p:\widehat{\mathbb{C}}=\mathbb{C}\cup\{\infty\}  \to \widehat{\mathbb{C}}$ such that every root of $p$ is an attracting fixed point  of $F_p$. Here a point $z$ is said to be an attracting fixed point of $F_p$ if $F_p(z)=z$ and $|F' _{p}(z)|<1$.	Let $F_p^n$ denote the $n$-times composition of $F_p$. For a root-finding method $F_p$, the sequence $\{F_p^n(z)\}_{n\geq 0}$ is supposed to converge to a root of $p$, at least in a neighborhood of the root. Any such point $z$ is often termed as an \textit{initial guess}. The challenge lies in making a right guess. The set of all $z \in \widehat{\mathbb{C}}$ for which $\{F_p^n(z)\}_{n\geq 0}$ converges to a root $z_0$ of $p$ is known as the basin of attraction of $z_0$.  Such a basin is always an open set but is not connected in general. The connected component of the basin of $z_0$ containing $z_0$ is known as the immediate basin of $z_0$. The union of all the basins corresponding to the roots of $p$ is precisely the set of right guesses. However, the complement of this set may contain an open set and that calls for a careful analysis of $\{F_p^n(z)\}_{n\geq 0}$ for all $z$. This is a primary motivation for studying the iteration of rational maps. 

\par Given a non-constant rational map $R$ with degree at least two, the extended complex plane $\widehat{\mathbb{C}}$ is partitioned into two sets, namely the Fatou set and the Julia set of $R$. The Fatou set, denoted by $\mathcal{F}(R)$, is the collection of all points where $\{R^n\}_{n\geq 0}$ is equicontinuous. The complement of $\mathcal{F}(R)$ is the Julia set, which we denote by $\mathcal{J}(R)$. By definition, $\mathcal{F}(R)$ is open and $\mathcal{J}(R)$ is closed. Further details on these sets can be found in ~\cite{Beardon_book}. The Julia set is usually a fractal set with complicated topology and finding them is highly non-trivial. If the Julia set of certain rational map is preserved by some simple M\"{o}bius map then it becomes possible to understand the structure of the Julia set through the M\"{o}bius map. 
A holomorphic Euclidean isometry of the plane is a map of the form $ z\mapsto az+b$ for  $ a, b \in \mathbb{C}$ with  $|a|=1$. The set of all such maps preserving the Julia set of a rational map $R$  is fundamental for the purpose of this article.  
\begin{definition}
$\Sigma R= \{\sigma: \sigma(z) = az+b ~\mbox{for}~ a , b \in \mathbb{C}, |a|=1~\mbox{ such that}~  \sigma(\mathcal{J}(R))=\mathcal{J}(R)\} $.
\end{definition}
Note that $\Sigma R$ is a group under composition of functions and is known as the symmetry group of $R$. It is said to be trivial if it contains the identity only. Each element of  $\Sigma R$ is called a symmetry of $\mathcal{J}(R)$. If $\Sigma R$ contains finite number of elements, then we define the order of $\Sigma R$ as the number of elements in $\Sigma R$ and we denote it as $o(\Sigma R)$. Note that every element of $\Sigma R$ permutes the Fatou components of $R$. The study of  $\Sigma R$ can be an indirect but useful way to understand $\mathcal{J}(R)$.  The main objective of this article is to understand the structure of the Julia set of a particular root-finding method, namely the Chebyshev's method.
\par  For a polynomial $p$, its Chebyshev's method, denoted by $C_p$ is defined as
\begin{equation}
C_{p}(z)=z-(1+\frac{1}{2}L_{p}(z))\frac{p(z)}{p'(z)},
\end{equation}
where $ L_{p}(z)=\frac{p(z)p''(z)}{[p'(z)]^{2}}
  $. This is a third order convergent method, i.e., the local degree of $C_p$ at each simple root of $p$ is at least $3$. For a linear polynomial or a monomial $p$, $C_p$ is either constant or a linear polynomial and, whenever $C_p$ is a linear polynomial, it is in fact a loxodromic M\"{o}bius map and every point except $\infty$ tends to $0$ under the iteration of $C_p$. For this reason, from now onwards, we consider polynomials that are not monomials and are of degree at least two.
Unlike the members of the well-studied K\"{o}nig's family, the Chebyshev's method can have an extraneous attracting fixed point (this is not a root of $p$). This may be a possible reason for which not much is known about its dynamics. Some results on the dynamics of Chebyshev's method  applied to unicritical polynomials $z \mapsto z^n +c$ can be found in \cite{CCV}. In \cite{Olivo2017}, Olivo et al. studied the extraneous fixed points of Chebyshev's method applied to some polynomials with real coefficients. The Julia sets of some root-finding methods, including Chebyshev's method are discussed by Kneisl \cite{Kneisl}. The article \cite{GV2020} reports some general properties of this method and proves the existence of attracting periodic points. The degree of $C_p$ is completely determined for all possible $p$ and the dynamics of $C_p$ is investigated for some cubic $p$ in   \cite{Nayak-Pal2022}.

\par	A map of the form $z \mapsto az+b$
for some $a,b \in \mathbb{C}, a \neq 0$ is called affine.
\begin{definition}
A root-finding method $F$ is said to satisfy the Scaling theorem if for every polynomial $p$, every $\lambda\in \mathbb{C}\setminus\{0\}$ and every affine map $T$, $T^{-1}\circ F_p \circ T=F_g$, where $g=\lambda p\circ T$. 
\end{definition}
Each  member of  K\"{o}nig's methods (see Lemma 8, \cite{BH2003}) and Chebyshev-Halley family (see Theorem 2.2, \cite{Nayak-Pal2022}) satisfies the Scaling theorem. Being a member of the  Chebyshev-Halley family, the Chebyshev's method satisfies the Scaling theorem. However,  there are methods like Stirling’s iterative method and Steffensen’s iterative method that do not satisfy the Scaling theorem (see \cite{ABP2004}). There is  an important relation between the symmetry group of $p$ and that of $F_p$ whenever $F_p$ satisfies the Scaling theorem. 
 The centroid $\xi_p $ of  a polynomial $p(z)=a_dz^d+a_{d-1}z^{d-1}+\dots +a_0$ is given by $ -\frac{a_{d-1}}{da_d}$. For every polynomial $p$, $\Sigma p$ is a rotation about $\xi_p$ (see Section 9.5,~\cite{Beardon_book}). A polynomial $p(z)=a_dz^d+a_{d-1}z^{d-1}+\dots +a_0$ is called monic or centered if its leading coefficient $a_d =1$ or its second leading coefficient $a_{d-1}=0$ respectively. A polynomial is called \textit{normalized} if it is monic and centered. We prove the following. 
  \begin{theorem}\label{symmetry-inclusion}
  	Suppose $F$ is a root-finding method satisfies the Scaling theorem. Then for a centered polynomial $p$ (which is not a monomial) with degree at least two, $\Sigma p \subseteq \Sigma F_p$.
  \end{theorem}
The above theorem is already known, but only in some special cases. A normalized polynomial $p$ can be written as $p(z)=z^\alpha p_0(z^\beta)$, where $p_0$ is a monic polynomial, $\alpha\in \mathbb{N}\cup \{0\}$ and $\beta\in \mathbb{N}$ are maximal for this expression. Then it is known that $\Sigma p=\{z\mapsto \lambda z: \lambda^\beta=1\}$ (Theorem 9.5.4 \cite{Beardon_book}). Yang considered the Newton method $N_p$ applied to a normalized polynomial $p$ and proved that $\Sigma p\subseteq \Sigma N_p$ (\cite{Yang2010}). He also proved that $\mathcal{J}(N_p)$ is a line if and only if $p(z)=c(z-a)^k(z-b)^k$, where $a,b,c\in \mathbb{C}$, $k\in \mathbb{N}$, $c\neq 0$ and $a\neq b$ (see Theorem 1.1 and 1.4, \cite{Yang2010}). Hence $\Sigma N_p$ contains a translation for this case. Liu and Gao prove that all these assertions of Yang are also true for every member of the K\"{o}nig's methods (see \cite{Liu-Gao2015}). Note that Theorem \ref{symmetry-inclusion} is not true in general if the polynomial is not taken to be centered (see Example \ref{example} in Section 2)
\par 
As the Chebyshev's method satisfies the Scaling theorem, we have the following consequence.
\begin{corollary}\label{Sym_C_p}
	For a centered polynomial $p$, $\Sigma p \subseteq \Sigma C_p$.
\end{corollary}
 A natural question is when the  equality holds. Since $\Sigma p$ does not contain any translation, the question of equality makes sense only when $\Sigma C_p$ does not contain any translation. This is in fact true.
  \begin{theorem}\label{no trans}
  	$\Sigma C_p$ does not contain any translation.
  \end{theorem}
Using a result of Boyd (Theorem 1,~\cite{Boyd2000}), it is seen that the Julia set of $C_p$ is a line whenever $\Sigma C_p$ contains a translation. Above theorem is proved by establishing that $\mathcal{J}(C_p)$ is never a line. This is done in Lemma~\ref{CTI} of this article.
\par It follows from Lemma~\ref{C1} (see Section 2) that each element  of  $\Sigma C_p$ is a rotation about a point $z_{C_p}$ depending on $C_p$.  Since $\Sigma p$ contains  rotations about its centroid,  the centroid $\xi_p$ must be $z_{C_p}$. Now, one may expect $\Sigma C_p =\Sigma p$!  This motivates a conjecture.
\begin{Conjecture}
If $p$ is a centered polynomial such that $\Sigma p$ is non-trivial then $\Sigma p =\Sigma  C_p $.
\label{conj}
\end{Conjecture}
The situation tends to be subtle when $\Sigma p$ consists of the identity only. In this case, the existence of rotations about a nonzero point in $\Sigma C_p$ is not ruled out.
We prove this conjecture for certain polynomials having non-trivial symmetry groups.

In the course of the proofs, the Fatou and the Julia set of $C_p$ are found and we have the following.  
 
  \begin{theorem}\label{equal sym}
  Let $p$ be a centered polynomial with degree at least two satisfying one of the following.
  \begin{enumerate}
  	\item $p$ has exactly two roots with the same multiplicity.
  	\item $p$ is unicritical.
  	\item $p$ is a cubic polynomial and $\Sigma p$ is non-trivial.
  	\item $p$ is a quartic polynomial, $0$ is a root of it and $\Sigma p$ is non-trivial.
  \end{enumerate}
Then $\Sigma p=\Sigma C_p$. 
Furthermore, the Fatou set $\mathcal{F}(C_p)$ is the union of all the attracting basins of $C_p$ corresponding to the roots of $p$ and the Julia set $\mathcal{J}(C_p)$ is connected in each of the above cases.
  \end{theorem}
The polynomial $p(z)=(z-1)^3 (z+1)^3$ is with exactly two roots with the same multiplicity. The Julia set  of its Chebyshev's method is shown as the common boundary of the yellow and the blue regions in Figure~\ref{unicrit_exacttwo}a. For the unicritical polynomial $p(z)=z^3-1$, the Julia set of $C_p$ is given as the common boundary of the yellow, blue and green regions in Figure~\ref{unicrit_exacttwo}b. 
Similarly, 	Figure~\ref{J_set_d3r2} shows the Julia set of $C_p$ where $p(z)=z(z^2-1)$ which is a cubic polynomial with non-trivial symmetry group. The case of quartic polynomials are given in Figures~\ref{J_set_d4r2} and \ref{J_set_d4r3}. It is important to observe that a region with a single colour is the basin (not the immediate basin of an attracting fixed point) and is not connected in any of these cases.
\par 
The article is organized in the following way.
Section $2$ makes some initial exploration of the symmetry group of root-finding methods satisfying the Scaling theorem and proves Theorem~\ref{symmetry-inclusion} along with some other useful results. The Julia sets and the symmetry groups of the Chebyshev's method applied to some polynomials, as enumerated in Theorem~\ref{equal sym} are determined in Section $3$. The article concludes with Section $4$ where few remarks and problems arising out of this work are stated. 
\par
All rational maps (including the polynomials)   considered in this article are of degree at least two. Also every polynomial considered is different from monomials unless stated otherwise. The boundary of an open connected subset $A$ of $\widehat{\mathbb{C}}$ is denoted by $\partial A$.
\section{ Scaling and symmetry}
For a rational map $R$, recall that $\Sigma R$  is the set of all holomorphic Euclidean isometries, in short Euclidean isometries that preserve $\mathcal{J}(R)$. This section discusses  possible structure of $\Sigma R$ when it contains at least one  non-identity element. The assumptions on $R$ are mostly those satisfied by the root-finding methods dealt with in this article. By saying a rotation or a translation, we mean non-trivial rotation or translation respectively where every map different from the identity  is called non-trivial. 
\par
The point $\infty$  is a superattracting fixed point of every polynomial (with degree at least two). Its Fatou set  contains  the immediate basin corresponding to $\infty$ which is an open set containing $\infty$. Therefore, the Julia set is bounded and the symmetry group of every polynomial does not contain any translation. But this is not true for rational maps in general. The point $\infty$ may not even be a fixed point. For example, the Julia set of the Newton method of $z^2-1$ is the imaginary axis which is invariant under every translation by a purely imaginary number.  Boyd proved that if $\mathcal{J}(R)$   is invariant under $z \mapsto z+1$, and the point at $\infty$ is either periodic or pre-periodic, then $\mathcal{J}(R)$ is either $\widehat{\mathbb{C}}$ or a horizontal line (Theorem 1, \cite{Boyd2000}). A minor improvement of this, is possible. For $z\in \mathbb{C}$ and $A\subset \mathbb{C}$, we define $z+A$ as $\{z+a:a\in A\}$.
\begin{lemma}\label{Boyd}
	Let $R$ be a rational map of degree at least two such that $\mathcal{J}(R)+a=\mathcal{J}(R)$ for some $a \neq 0$. If $\infty$ is either periodic or pre-periodic then $\mathcal{J}(R)$ is the extended complex plane or a line.
\end{lemma}
\begin{proof}
	Consider a rational map $f=\phi \circ R \circ \phi^{-1}$, where $\phi(z)=\frac{z}{a}$. Then $\infty$ is either periodic or pre-periodic for $f$. Note that
	$\mathcal{J}(f)  =\phi (\mathcal{J}(R))$. By the hypothesis, $  \phi(\mathcal{J}(R)) = \phi(\mathcal{J}(R)+a)$, which is nothing but $\mathcal{J}(f)+1.
$
	Hence by Boyd's result (Theorem1, \cite{Boyd2000}), $\mathcal{J}(f)$ is either the whole extended complex plane or a horizontal line. If $\mathcal{J}(f)$ is $\widehat{\mathbb{C}}$, then $\mathcal{J}(R)=\widehat{\mathbb{C}}$. If $\mathcal{J}(f)$ is a horizontal line  then $\mathcal{J}(R)$ is a line with slope $\tan(Arg(a))$.
\end{proof} 
The above lemma gives under some situation that the  Julia set is not complicated whenever there is a translation in the symmetry group of the map. Now we consider the other case, i.e., when the symmetry group does not contain any translation. 
\par
  A rotation about any point in $\mathbb{C}$ is an Euclidean isometry. If $ \sigma_1, \sigma_2 \in \Sigma R$ are non-trivial rotations  about two different points $\alpha$ and $\beta$ in $\mathbb{C}$ respectively then $\sigma_1 \circ \sigma_2 \circ \sigma_1 ^{-1} \circ \sigma_2 ^{-1}$  is a non-trivial translation. In fact, if this $\sigma_i$ is a rotation by angle $\theta_i \in (0,2\pi), i=1,2$ then $\sigma_1 \circ \sigma_2 \circ \sigma_1 ^{-1} \circ \sigma_2 ^{-1}(z)=z+ (\beta-\alpha)(e^{i\theta_1}+e^{i\theta_2}-e^{i(\theta_1 +\theta_2)}-1)=z+ (\beta-\alpha)(e^{i\theta_1}-1)(1-e^{i\theta_2})$ and the constant term is always nonzero. This observation gives the following.
\begin{lemma}\label{C1}
	If the Julia set of a rational map $R$ is not invariant under any non-trivial translation and $\Sigma R$ contains at least one non-identity element then there exists $z_R \in \mathbb{C}$ such that every element of $\Sigma  R $ is a rotation about $z_R$.
\end{lemma}
Let $\sigma\in \Sigma R$ fix  an attracting fixed point $z_0$  of $R$. If $z_0 \in \mathbb{C}$ then $\sigma$ is not a translation. Otherwise, i.e., if $z_0 =\infty$ then $\mathcal{J}(R)$ is bounded and $\sigma$ cannot be a translation. This situation is dealt with in the next lemma.
\begin{lemma}\label{imm_basin_preserved}
	Let $R$ be rational map of degree at least two and $\sigma\in \Sigma R$. If $\sigma$ fixes a (super)attracting fixed point $z_0$ of $R$ then $\sigma(\mathcal{A}_{z_0})=\mathcal{A}_{z_0}$, where $\mathcal{A}_{z_0}$ is the immediate basin of $z_0$.
\end{lemma}
\begin{proof}
	As $\sigma\in \Sigma R$, $\sigma$ preserves the Julia set of $R$. Therefore a Fatou component of $R$ is mapped onto a Fatou component by $\sigma$. As $\sigma(z_0)=z_0$, $\mathcal{A}_{z_0}$ and $\sigma(\mathcal{A}_{z_0})$ intersect. This proves that $\sigma(\mathcal{A}_{z_0})=\mathcal{A}_{z_0}$. 
\end{proof}
Recall that if a root-finding method $F$ satisfies the Scaling theorem then $F_p$ is affine conjugate to $F_{\lambda p \circ T}$ via the affine map $T$. In order to prove Theorem~\ref{symmetry-inclusion}, we need to relate their symmetry groups. Here is a result for this purpose.
\begin{lemma}\label{Cg}
	If for two rational maps $R$ and $S$, there is an affine map $\gamma$ such that $S=\gamma \circ R \circ  \gamma^{-1}$ then $\Sigma  S = \gamma (\Sigma  R) \gamma^{-1}$.
\end{lemma}
\begin{proof}
	By Theorem 3.1.4, \cite{Beardon_book}, $\mathcal{J}(S)=\gamma (\mathcal{J}(R)).$ If $\sigma \in \Sigma  R $, where $\sigma (z)=az+b$, and $\gamma(z)=\alpha z+\beta$, $a,b,\alpha,\beta\in \mathbb{C}, |a|=1, \alpha\neq 0$ then $\gamma \sigma \gamma^{-1}(z)=az-a\beta+\alpha b+\beta$ is an Euclidean isometry. Further,
	\begin{align*}
	\gamma \sigma \gamma^{-1}(\mathcal{J}(S))& =\gamma \sigma (\mathcal{J}(R))\\
	& = \gamma (\mathcal{J}(R))=\mathcal{J}(S)
	\end{align*}
	This implies that $\gamma \circ \sigma \circ  \gamma^{-1}\in \Sigma S.$ Therefore, $\gamma (\Sigma R) \gamma^{-1} \subseteq \Sigma S.$ Similarly, it can be shown that $\Sigma  S  \subseteq \gamma (\Sigma R) \gamma^{-1}.$
\end{proof}
Recall that the centroid $\xi_p $ of $p(z)=a_dz^d+a_{d-1}z^{d-1}+\dots +a_0$  is given by $ -\frac{a_{d-1}}{da_d}$.
Note that $p \circ T$ is centered and its leading coefficient is $a_d$ where $T(z)=z+\xi_p$. Then $g=\frac{1}{a_d}p\circ T$ is a normalized polynomial. Now, if any root-finding method $F_p$ satisfies the Scaling theorem then  $F_p=T \circ F_g \circ T^{-1} $. It follows from Lemma~\ref{Cg} that  $\Sigma  F_p=T (\Sigma  F_g)  T^{-1} $. There is a useful observation.
\begin{Observation}\label{normalization_1}
Let $p$ be a centered polynomial and $F$ be a root-finding method satisfying the Scaling theorem.
\begin{enumerate}
	\item If the leading coefficient of $p$ is $a_d$ then $g=\frac{1}{a_d}p$ is a normalized polynomial and $\Sigma p=\Sigma g$. As $F$ satisfies the Scaling theorem, $F_p=F_g$ and thus  $\Sigma  F_p=\Sigma  F_g$. Thus,  it is enough to consider normalized polynomials in order to analyze the symmetry groups and dynamics of root-finding methods satisfying the Scaling theorem.
	\item Consider $h=p\circ T$ where $T(z)=Az$, $A\in \mathbb{C}\setminus \{0\}$. Then $h$ is also a centered polynomial and $\Sigma p=\Sigma h$. Also, from the Scaling theorem, we get $F_h=T^{-1}\circ F_p\circ T$, and thus $\Sigma F_h=T^{-1}(\Sigma F_p) T=\Sigma F_p$.
\end{enumerate}
\end{Observation}
Now we present the proof of Theorem \ref{symmetry-inclusion}.
\begin{proof}[Proof of Theorem \ref{symmetry-inclusion}]
	By Observation~\ref{normalization_1}(1) let $p$ be a normalized polynomial and  $\sigma \in \Sigma p$. Then by Lemma 9.5.6,~\cite{Beardon_book},  $p \circ \sigma=\sigma^d \circ p$ where   $\deg(p)=d$. Since $\sigma(z)=\lambda z$ and the root-finding method satisfies the Scaling theorem, $F_{p \circ \sigma} =F_p $. Further, it follows that  $\sigma \circ F_{p \circ \sigma} \circ \sigma^{-1} =F_p $ and hence $\sigma \circ F_{p } \circ \sigma^{-1} =F_p $ . By Theorem 3.1.4,~\cite{Beardon_book}, $\sigma (\mathcal{J}(F_p))=\mathcal{J}(F_p).$ Hence, $\sigma \in \Sigma F_p.$ 
\end{proof}
Theorem \ref{symmetry-inclusion} is not necessarily true for a polynomial which is not centered.
\begin{example}\label{example}
	Consider the polynomial $p(z)=z^2+2z$. Then its centroid is $-1$ and it is conjugate to $z^2$. Thus the Julia set of $p$ is a circle centered at $-1$ and hence $\Sigma p=\{z\mapsto \lambda(z+1)-1: |\lambda|=1\}$ is an infinite set.
	\par 
	Now consider $g(z)=p(z-1)=z^2-1$. By the Scaling theorem, $C_g=T^{-1}\circ C_p\circ T$, where $T(z)=z-1$. Again by Lemma \ref{Cg}, $\Sigma C_g=T^{-1}(\Sigma C_p)T$. Note that $\Sigma C_g=\Sigma g=\{I,z\mapsto -z\}$ (see the proof of Theorem \ref{equal sym}). Thus $\Sigma C_p=\{I, z\mapsto -z-2\}$ which does not contain $\Sigma p$.
\end{example}
For two polynomials $p$ and $g$, we say that $\Sigma p$ and $\Sigma g$ are isomorphic if either both the sets are infinite or $o(\Sigma p)=o(\Sigma g)$. Observe that in the above example $\Sigma p$ is an infinite set, whereas $\Sigma g$ is finite, i.e., $\Sigma p$ and $\Sigma g$ are not isomorphic.
\begin{Remark}\label{equal_condition}
	Any polynomial $p$ with centroid $\xi$($\neq 0$) can be transformed into a centered polynomial $g$ by considering $g=p\circ T$ where $T(z)=z+\xi$. Note that $\Sigma p=T(\Sigma g)T^{-1}$ whenever $\Sigma p$ and $\Sigma g$ are isomorphic.
	\par 
	If a root-finding method $F$ satisfies the Scaling theorem then by Theorem \ref{symmetry-inclusion}, $\Sigma g\subseteq \Sigma F_g$ (we exclude the possibility that $g$ is a monomial). The Scaling theorem gives that $F_g=T^{-1}\circ F_p\circ T$ and by Lemma \ref{Cg}, we get $\Sigma F_p=T(\Sigma F_g)T^{-1}$. Thus, if $\Sigma g$ is non-trivial then $\Sigma F_p$ contains rotations about the centroid of $p$, i.e., $\xi$. If $\Sigma p$ and $\Sigma g$ are isomorphic (i.e., $o(\Sigma p)=o(\Sigma g)$ whenever $\Sigma g$ is finite) then the relation $\Sigma g\subseteq \Sigma F_g$ gives that $\Sigma p\subseteq \Sigma F_p$. Whether this inclusion holds good if $\Sigma p$ and $\Sigma g$ are not isomorphic remains to be explored.
\end{Remark} 
We conclude with an useful lemma on the symmetry group of root-finding methods.
\begin{lemma}
	Let $p$ be a normalized polynomial with non-trivial symmetry group and $F$ is a root-finding method satisfying the Scaling theorem such that $F_p$ have an unbounded Fatou component $\mathcal{A}$ not containing $0$. If $\Sigma F_p$ does not contain any non-trivial translation and  every unbounded Fatou component of it not containing  $0$ is the $\gamma$-image of $\mathcal{A}$ for some $\gamma \in \Sigma p$  then $\Sigma F_p = \Sigma p$. 
	\label{symmetry-equality}
\end{lemma}
\begin{proof}
	By Theorem~\ref{symmetry-inclusion}, $\Sigma p \subseteq \Sigma F_p$. In order to prove the equality, let $\sigma \in \Sigma F_p$ be non-identity.
	There is no translation in  $\Sigma {F_p}$  by assumption. It follows from Lemma~\ref{C1} that every element of $F_p$ is a rotation about the origin. Since $\sigma$ maps an unbounded Fatou component not containing the origin onto an unbounded Fatou component not containing the origin, $\sigma(\mathcal{A})$ is an unbounded Fatou component not containing the origin. By the assumption, $\sigma(\mathcal{A})= \gamma(\mathcal{A})$ for some $\gamma \in \Sigma p$. As $\gamma$ and $\sigma$ are two analytic maps agreeing on a domain $\mathcal{A}$, $\gamma=\sigma$. In other words, $\sigma \in \Sigma p$.
\end{proof}
\section{The Chebyshev's method} The proofs of Theorems~\ref{no trans} and \ref{equal sym} are to be provided in this section. \par 
Recall that the Chebyshev's method of a polynomial $p$  is defined as
\begin{equation}
C_{p}(z)=z-(1+\frac{1}{2}L_{p}(z))\frac{p(z)}{p'(z)},
\end{equation} where  $ L_{p}(z)=\frac{p(z)p''(z)}{[p'(z)]^{2}}
$. 
 The derivative of the Chebyshev's method is given by
\begin{equation}\label{deri}
C_p'(z)=\frac{L_{p}(z)^2}{2}(3-L_{p'}(z))
\end{equation}
where $L_{p'}(z)=\frac{p'(z)p'''(z)}{[p''(z)]^2}$. 
 That $C_p$ satisfies the Scaling theorem leads to a significant amount of simplification. 
\begin{Observation}\label{appli_ST}
\begin{enumerate}
\item Let $a$ and $b$ be two distinct roots of $p$ with multiplicities $k$ and $m$ respectively. Thus, $p$ is of the form $p(z)=(z-a)^k(z-b)^mp_1(z)$ where $p_1$ is a polynomial such that $p_1(a),p_1(b)\neq 0$. Then by post-composing $p$ with the affine map $T(z)=\frac{a-b}{2}z+\frac{a+b}{2}$, we get $$g(z)=p(T(z))=c(z-1)^k(z+1)^m p_1(T(z))$$ where $c$ is a nonzero constant. Since $C_p$ satisfies the Scaling theorem, we get $C_g=T^{-1} \circ C_p \circ T$. Hence, any two distinct roots of a polynomial can be taken as $-1$ and $1$.
\item Let the polynomial $q$ be of the form $q(z)=(z^n-a)^kq_1(z)$, where $n\geq 2$, $k\in \mathbb{N}$, $a=re^{i\theta}$ for $ r>0, \theta\in [0,2\pi) $, and not all the $n$-th roots of $a$ are the roots of $q_1$. Consider the map $h=q\circ \alpha$ where $\alpha(z)=r^{\frac{1}{n}}e^{\frac{i\theta}{n}}z$. Then $h(z)=a^k(z^n-1)^k q_1(r^{\frac{1}{n}}e^{\frac{i\theta}{n}}z)$. Since $C_p$ satisfies the Scaling theorem, the maps $C_q$ and $C_h$ are affine conjugate. Hence, for such a polynomial, without loss of generality we can consider $a=1$.
	\end{enumerate}
\begin{Remark}\label{normalization}
	From Observation \ref{normalization_1}(2) we deduce the following particular case of the polynomial $p$ considered in Observation \ref{appli_ST}(2). If $p(z)=\lambda z^\alpha(z^n-a)^k$ for some $\lambda,~ a\in \mathbb{C}\setminus \{0\}$ and $\alpha\in \mathbb{N}\cup\{0\}$, $n,~k\in \mathbb{N}$, then $p(a^{\frac{1}{n}}z)=\lambda a^{\frac{\alpha}{n}+k} z^\alpha(z^n-1)^k$, where we consider any $n$-th root of $a$. Note that $\Sigma p=\Sigma \tilde{p}$, where $\tilde{p}(z)=z^\alpha(z^n-1)^k$ by Observation \ref{normalization_1}(2). As the Chebyshev's method applied to a polynomial is the same as that applied to a constant multiple of it, consequently we get, $\Sigma C_p=\Sigma p$ if and only if $\Sigma C_{\tilde{p}}=\Sigma \tilde{p}$.
\end{Remark}
\end{Observation}
Now we investigate the possible elements of $\Sigma C_p$. The immediate question is that whether $\Sigma C_p$ contains a translation? Since for each polynomial $p$ of degree at least two has a finite root, the map $C_p$ has an attracting fixed point. Therefore, $\mathcal{J}(C_p)$ can not be $\widehat{\mathbb{C}}$. Also, $\infty$ is always a repelling fixed of $C_p$ (see Proposition 2.3,  \cite{Nayak-Pal2022}). It follows from Lemma \ref{Boyd} that   if $\Sigma  C_p$ contains a translation then $\mathcal{J}(C_p)$ is itself a line. Therefore, to answer the above question, it is enough to investigate the possibility of $\mathcal{J}(C_p)$ to be a line.
\begin{lemma}\label{CTI}
	For each polynomial $p$, $\mathcal{J}(C_p)$ is never a line.
\end{lemma}
\begin{proof}
Suppose on the contrary that the Julia set of $C_p$ is a line for some polynomial $p$. Then $C_p$ has two Fatou components and those are the only Fatou components. If $p$ has at least three distinct roots then $\mathcal{F}(C_p)$ contains at least three attracting fixed points leading to at least three Fatou components, which can not be true. Therefore $p$ has exactly two distinct roots $a$ and $b$ with multiplicities, say $k$ and $m$ respectively, where $k,m\geq 1$. Since $C_p$ satisfies the Scaling theorem, without loss of generality assume that $p$ is monic. In view of Observation \ref{appli_ST}(1), consider $a=1$ and $b=-1$. Thus
	\begin{equation}
	p(z)=(z-1)^k(z+1)^m.
	\end{equation}
	Then $$L_p(z)=\frac{(k+m)(k+m-1)z^2+2(k-m)(k+m-1)z+(k-m)^2-(k+m)}{[(k+m)z+(k-m)]^2}$$ and
	\begin{equation}
	C_p(z)=z-\dfrac{(z-1)(z+1)f(z)}{2\{(k+m)z+(k-m)\}^3},
	\label{minimumdegree-C_p}
	\end{equation}
	where
	 \begin{equation}\label{extra_f}
	 f(z)=(k+m)(3k+3m-1)z^2+2(k-m)(3k+3m-1)z+3(k-m)^2-(k+m).
	 \end{equation}
	The extraneous fixed points of $C_p$ are the solutions of the Equation \ref{extra_f}.	The discriminant of this quadratic equation is 
	$4(k-m)^2(3k+3m-1)^2-4(k+m)(3k+3m-1)[3(k-m)^2-(k+m)] =16km(3k+3m-1)$ which is positive.
	The roots of $f(z)$ are real and those are
	\begin{align}
	& \alpha=\dfrac{-(k-m)(3k+3m-1)+ 2\sqrt{km(3k+3m-1)}}{(k+m)(3k+3m-1)}\label{extra1}\\  and ~~
	& \beta=\dfrac{-(k-m)(3k+3m-1)- 2\sqrt{km(3k+3m-1)}}{(k+m)(3k+3m-1)}\label{extra2}.
	\end{align}
	As an extraneous fixed point $z$ is a solution of Equation \ref{extra_f}, i.e., $f(z)=0$,  the multiplier of $z$ is \begin{align*}
	\lambda(z)&=1-\frac{1}{2}\frac{(z-1)(z+1)f'(z)}{[(k+m)z+(k-m)]^3}\\
	&=1-\frac{(z-1)(z+1)(3k+3m-1)}{[(k+m)z+(k-m)]^2}.
	\end{align*}
	Thus
	\begin{align*}
	& \lambda(\alpha)=1+\dfrac{(3k+3m-1)^2}{km(k+m)^2}\left[km\left(\dfrac{3k+3m-2}{3k+3m-1}\right)+(k-m)\sqrt{\dfrac{km}{3k+3m-1}}\right]~\mbox{and}~\\
	& \lambda(\beta)=1+\dfrac{(3k+3m-1)^2}{km(k+m)^2}\left[km\left(\dfrac{3k+3m-2}{3k+3m-1}\right)-(k-m)\sqrt{\dfrac{km}{3k+3m-1}}\right].
	\end{align*}
	If $k=m$ then both the multipliers are the same and it is bigger than $1$. Without loss of generality assuming $k>m$, it is clear that $\lambda(\alpha)>1$. To see $\lambda(\beta)>1$, first note that $(3k+3m-1) \left\{k^2 (m-1)+m^2(k-1)+km(2k+2m-1)\right\}+km>0 $ and the left hand side expression is  $(3k+3m-1)(3k^2 m+3k m^2-k^2-m^2-km)+km $ which is equal to $ (3k+3m-1)\left\{ km(3k+3m-1)-2km +2km-k^2 -m^2   \right\}+km$.  This is nothing but $(3k+3m-1)\left\{ km(3k+3m-1)-2km -(k-m)^2   \right\}+km $. Now rearranging the terms, we get that $km\left\{ (3k+3m-1)^2 -2(3k+3m-1)+1\right\} -(k-m)^2 (3k+3m-1)>0$ i.e., $km(3k+3m-2)^2 > (k-m)^2 (3k+3m-1)$. Thus $km\left(\dfrac{3k+3m-2}{3k+3m-1}\right)-(k-m)\sqrt{\dfrac{km}{3k+3m-1}}>0 $.  
	Therefore,  the two extraneous fixed points are repelling. 
	Since these are in the Julia set of $C_p$, the Julia set must be the real line whenever it is a line. However $C_p$ has real attracting fixed points, namely $1$ and $-1$ and that are in the Fatou set. This leads to a contradiction and the proof concludes.
\end{proof}
Now the proof of the Theorem \ref{no trans} becomes straightforward.
\begin{proof}[The proof of Theorem \ref{no trans}]
	As $\mathcal{J}(C_p)$ is not  a line for any $p$, $\Sigma C_p$ does not contain any translation.
\end{proof}
We require six lemmas for proving  Theorem~\ref{equal sym}. The first two deal with rational maps with an unbounded invariant attracting immediate basin.
\begin{lemma}
	Let $R$ be a rational map with a repelling fixed point at $\infty$. If $\mathcal{A}$  is an unbounded and invariant attracting domain of $R$ and $U$ is a Fatou component different from $\mathcal{A}$ such that $R(U) =\mathcal{A}$ then $U$ is bounded. Furthermore, if $U'$ is a Fatou component such that $R^k (U')=U$ for some $k>0$ then $U'$ is bounded.
	\label{bounded}
\end{lemma}
\begin{proof}
	Let $N$ be a neighborhood of $\infty$ in which $R$ is one-one and $N \subsetneq R(N)$. This is possible since $\infty$ is a repelling fixed point of $R$. Let $g$ be the branch of $R^{-1}$ such that $g(R(N))=N$, i.e., $g(R(z))=z$ for all $z\in N$. Let $z_0 \in N \cap \mathcal{A}$. Then $z_1=R(z_0) \in R(N) \cap \mathcal{A}$. Clearly $g(z_1)=z_0$. Let $z' $ be an arbitrary point in $R(N) \cap \mathcal{A}$ different from $z_1$. Consider a Jordan arc $\gamma$ in $\mathcal{A}$ joining $z_1$ and $z'$ such that $g$ is continued analytically along $\gamma$ as a single valued analytic function by the Monodromy theorem. Further, $g(z') \in N $. Since $g(\gamma)$ is in the Fatou set and $g(z_1)=z_0 \in \mathcal{A}$, it is in a single Fatou component, which is nothing but $\mathcal{A}$. Thus $g(z') \in \mathcal{A}$ gives that $g(z') \in N \cap \mathcal{A}$. In other words, $g$ maps $R(N) \cap \mathcal{A}$ onto $N \cap \mathcal{A}$. 
	\par 
	Let $U$ be a Fatou component of $R$ different from $\mathcal{A}$ such that  $R(U)=\mathcal{A}$. If  $z \in U$ and $w \in R(N) \cap \mathcal{A} $ such that $R(z)=w$ then $z \notin N$. This follows from the conclusion of the previous paragraph and the fact that $R$ is one-one in $N$. In other words, $U \cap R^{-1}( R(N) \cap \mathcal{A})$ is bounded. Now $\mathcal{A} \setminus R(N)$ is bounded. Since the pre-image of every bounded set under $R$ is bounded (since the image of each unbounded set under $R$ is unbounded), the pre-image $R^{-1}(\mathcal{A} \setminus R(N))$ is a bounded set.  
	This implies that $ U \cap R^{-1}(\mathcal{A} \setminus R(N)) $ is bounded. Therefore $U=[U \cap R^{-1}( R(N) \cap \mathcal{A})]\cup [ U \cap R^{-1}(\mathcal{A} \setminus R(N))]$ is bounded.
	\par Being a repelling fixed point, $\infty$ is in the Julia set of $R$. It is on the boundary of each unbounded Fatou component of $R$. Further, the $R^k$-image of every unbounded Fatou component contains $\infty$ on its boundary (since $\infty$ is fixed by $R^k$ and $R^k$ maps the boundary of a Fatou component $V$ onto the boundary of $R^k(V)$) for all $k>0$. Since $U$ is bounded, $U'$ is bounded.
\end{proof}
\begin{lemma}[Lemma 4.3, \cite{Nayak-Pal2022}]\label{Nayak_Pal}
	If $R$ is a rational map with $\infty$ as a repelling fixed point and $\mathcal{A}$ is an invariant, unbounded immediate basin of attraction then $\partial \mathcal{A}$ contains at least one pole. Moreover, if $\mathcal{A}$ is simply connected and $\partial \mathcal{A}$ contains all the poles then $\mathcal{J}(R)$ is connected.
\end{lemma}
The next result provides  a condition that ensures a connected Julia set.
\begin{lemma}\label{connected J'set}
	Let $R$ be a rational map and $\infty$ be a repelling fixed point of $R$. If the unbounded Julia component of $R$ contains all the poles of $R$ then $\mathcal{J}(R)$ is connected. Consequently, all the Fatou components are simply connected.
\end{lemma}
\begin{proof}
On the contrary, suppose that $\mathcal{J}(R)$ is not connected. Then a simple closed curve $\gamma$ in the Fatou set can be chosen such that the bounded component of its complement contains a  point of $ \mathcal{J}(R)$. In this case, we say that $\gamma$ surrounds a point of the Julia set. Since the Julia set is the closure of the backward orbit of any  point in it (see Theorem 4.2.7, \cite{Beardon_book}), there exists a $k$ such that $R^k(\gamma)$ is a closed and bounded curve (not necessarily simple) lying in the Fatou set which surrounds a pole, say $\tau$ of $R$. Thus, $\infty$ is separated from $\tau$ by a closed curve lying in the Fatou set. That contradicts our assumption that the unbounded Julia component contains all the poles of $R$. Hence the Julia set is connected and from Theorem 5.1.6, \cite{Beardon_book}, we conclude that all the Fatou components are simply connected. 
\end{proof}
We say a rational map $R=\frac{P}{Q}$ is with real coefficients if all the coefficients of the polynomials $P$ and $Q$ are real numbers.
\begin{lemma}\label{invariant}
	If a  rational map is odd and with real coefficients, then its Julia set (and thus the Fatou set) is preserved under $z \mapsto -\bar{z}$.
\end{lemma}
\begin{proof}
	As $R$ is odd,  $R(-z)=-R(z)$ for all $z$. Further, since all the coefficients of $R$ are real, $\overline{R(\bar{z})}=R (z)$ for all $z$. Thus $\phi \circ R \circ \phi^{-1}=R$ where $\phi(z)=-\bar{z}$. It now follows from Theorem 3.1.4,~\cite{Beardon_book} that $\phi (\mathcal{J}(R))=\mathcal{J}(R)$.
\end{proof}
We need a well-known result (for example see Lemma 4.1,~\cite{Nayak-Pal2022}) concerning the relation of Fatou components of a rational map and its critical points. 
\begin{lemma}
	Let $U$ be a periodic Fatou component of  a rational map $R$ and $C_R$ be the set of all its critical points.
	\begin{enumerate}
		\item 	If $U$ is an immediate attracting basin or an immediate parabolic basin  then  $U \cap C_R \neq \emptyset$.
		\item 	If $U$ is a Siegel disk or a Herman ring  then its boundary is contained in the closure of  $\{R^n(c):n \geq 0~\mbox{and}~c \in C_R \}$.
	\end{enumerate}
	\label{basic-dynamics-lemma}
\end{lemma}
An elementary result is also required to be used frequently.
\begin{lemma}
	\begin{enumerate}
		\item Let $f: (-\infty,a] \to \mathbb{R}$ be such that $f'(x)>0$ and $f(x)>x$ for each $x< a$. If $f(a)=a$  then $\lim\limits_{n \to \infty}f^n(x)=a$ for all $x  \leq a$.
		\item Let $f: [b,\infty) \to \mathbb{R}$  be such that $f'(x)>0$ and $f(x)< x$ for each $x > b$. If $f(b)=b$  then $\lim\limits_{n \to \infty}f^n(x)=b$ for all $x  \geq b$.
	\end{enumerate}
	\label{monotone}
\end{lemma}
We are now in a position to prove the Conjecture~\ref{conj} in certain cases.
\subsection{Polynomial having two distinct roots with the same multiplicity}
A polynomial $p$ having exactly two roots with the same multiplicity is of the form $p(z)=\lambda(z-a)^k(z-b)^k$ where $\lambda\in\mathbb{C}\setminus \{0\}$, $a,b\in \mathbb{C}$ and $a\neq b$. Note that $p$ is centered if and only if $a=-b$ i.e., $p(z)=\lambda(z^2-a^2)^k$. In this case the symmetry group of Julia set of $p$ contains exactly two elements, namely rotations of order two about the origin (Theorem 9.5.4 \cite{Beardon_book}).
\begin{proof}[Proof of the Theorem \ref{equal sym}(1)]
	In view of Remark \ref{normalization}, without loss of generality, assume that $p(z)=(z^2-1)^k$. In this case $\Sigma p=\{I, z\mapsto -z\}\subseteq \Sigma C_p$ (by Corollary \ref{Sym_C_p}). The points $1$ and $-1$ are the attracting fixed points of $C_p$ and let $\mathcal{A}_1$ and $\mathcal{A}_{-1}$ be their respective immediate basins.
\par  
	If the roots of $p$ are simple, i.e., $k=1$ then $$C_p(z)=\frac{3z^4+6z^2-1}{8z^3}~\mbox{and}~ C_p'(z)=\frac{3(z^2-1)^2}{8z^4}.$$
The point $0$ is a multiple pole and therefore is a critical point of $C_p$ with multiplicity $2$. The only other critical points are the roots of $p$, i.e., $\pm 1$ with multiplicity $2$ each.
	 Note that  $C_p'(x)>0$  and $C_p(x)-x=-\frac{(x^2-1)(5x^2-1)}{8x^3}>0$ for all $x  < -1$. Therefore, $\lim\limits_{n \to \infty}C_p^n(x)=-1 $ for all $x \leq -1$ by Lemma~\ref{monotone}(1). In other words, $(-\infty, -1] \subset \mathcal{A}_{-1}$. Since the map $z\mapsto -z$ preserves  $\mathcal{F}(C_p)$, $[1,\infty)\subset \mathcal{F}(C_p)$. In fact $[1,\infty)\subset \mathcal{A}_1$.
	 Since there is no critical point in $\mathcal{F}(C_p)$ other than $\pm 1$, $\mathcal{F}(C_p)$ is the union of the basins of $1$ and $-1$. This follows from Lemma~\ref{basic-dynamics-lemma}.
In other words, if $U'$ is a Fatou component of $C_p$ then there is a $k \geq 0$ such that $C_p ^k (U')=\mathcal{A}_1$ or $\mathcal{A}_{-1}$. Now it follows from Lemma~\ref{bounded} that $U'$ is bounded.
Thus $\mathcal{A}_{\pm 1}$ are the only unbounded Fatou components of $C_p$. Note that $\sigma(\mathcal{A}_1)=\mathcal{A}_{-1}$ where $\sigma(z)=-z$. It follows from  Theorem~\ref{no trans} and Lemma~\ref{symmetry-equality} that $\Sigma C_p =\Sigma p$.
\par
If $k\geq 2$ then, by Equation \ref{minimumdegree-C_p}, $$C_p(z)=\frac{(2k-1)(4k-1)z^4+6kz^2-1}{8k^2z^3}$$ and 
$$C_p'(z)=\frac{(2k-1)(4k-1)z^4-6kz^2+3}{8k^2z^4}.$$
Note that $0$ is a multiple pole and hence a critical point lying in the Julia set of $C_p$.
The nonzero critical points of $C_p$ are the solutions of $(2k-1)(4k-1)z^4-6kz^2+3=0$. Taking $z^2=w$, it is seen that $w=\frac{3k\pm \sqrt{-3(5k-1)(k-1)}}{(2k-1)(4k-1)} $. As $k>1$, $w$ is non-real and so also its square roots. Now the set of all the nonzero critical points of $C_p$ is $\{c, -c, \bar{c}, -\bar{c}\}$ for some $c \in \mathbb{C} \setminus \mathbb{R}$.
There is a critical point, say $c$  in $\mathcal{A}_1$ (see Lemma \ref{basic-dynamics-lemma}). It follows from Lemma~\ref{invariant} that $\mathcal{A}_1$ is symmetric with respect to the real line and therefore it contains $\bar{c}$. It also follows from the same lemma that $-c, -\bar{c} \in \mathcal{A}_{-1}$.
\par  
As $k\geq 2$, $C_p'(x)=\frac{(2k-1)(4k-1)x^4-6kx^2+3}{8k^2x^4}>\frac{6kx^2(x^2-1)+3}{8k^2x^4}>0$ for all $x \leq -1$. Also, $C_p(x)-x=-\frac{(x^2-1)\{(6k-1)x^2-1\}}{8k^2x^3}>0$ gives that $C_p(x)>x$ for $x <  -1$. Using the same argument as in $k=1$ case, it is found that  $ (-\infty,-1] \subset \mathcal{A}_{-1}$ and $[1, \infty) \subset \mathcal{A}_{1} $. That   $\mathcal{F}(C_p)$ is the union of the basins of $\pm 1$, $\mathcal{A}_{\pm 1}$ are the only unbounded Fatou components of $C_p$ and consequently $\Sigma p = \Sigma C_p$ follow also from the same arguments. 
\par Now, we are to prove the connectedness of the Julia set for each 
$k \geq 1$.   The point $0$ is the only pole of $C_p$ and is on the boundary of both $\mathcal{A}_{1}$ and  $\mathcal{A}_{-1}$ by Lemma~\ref{Nayak_Pal}. The unboundedness of $\mathcal{A}_{1}$ and  $\mathcal{A}_{-1}$ gives that the Julia component containing $0$ is unbounded. Now it follows from Lemma~\ref{connected J'set} that   $\mathcal{J}(C_p)$ is connected.
\end{proof}
\begin{Remark}
\begin{enumerate}
		\item For $k=1$, the extraneous fixed points are $\pm \frac{1}{\sqrt{5}}$ and that are repelling with multiplier $6$ each (see Equations \ref{extra1}, \ref{extra2}). For $k >1$, it follows from the proof of Lemma~\ref{CTI} that all the extraneous fixed points are repelling.
 \item The imaginary axis is forward invariant under $C_p$  and therefore it is contained in $\mathcal{J}(C_p)$ in the theorem above. 
\end{enumerate}
\end{Remark}

\begin{figure}[h!]
	\begin{subfigure}{.5\textwidth}
		\centering
		\includegraphics[width=0.9\linewidth]{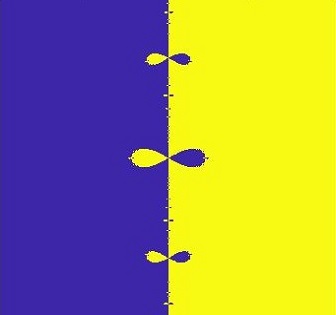}
		\caption{$\mathcal{J}(C_p)$ for $p(z)=(z-1)^3(z+1)^3$}
	\end{subfigure}
	\begin{subfigure}{.5\textwidth}
		\centering
		\includegraphics[width=0.9\linewidth]{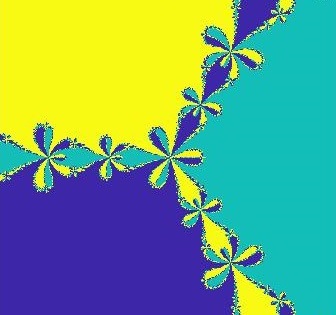}
		\caption{$\mathcal{J}(C_p)$ for $p(z)=z^3-1$}
	\end{subfigure}
	\caption{The Julia sets of $C_p$; (a) $p$ has exactly two roots with the same multiplicity, (b) $p$ is unicritical }
	\label{unicrit_exacttwo}
\end{figure}

\subsection{The unicritical polynomials}
We are going to prove Theorem \ref{equal sym}(2).
 
\begin{proof}[Proof of Theorem \ref{equal sym}(2)]
Every unicritical and centered polynomial is of the form $\lambda z^n +c$ for some $\lambda,c \in \mathbb{C}\setminus \{0\}$ and $n \geq 2$. However, by Remark \ref{normalization}, without loss of generality we consider $p(z)=z^n-1$ for some $n \geq 2$. If $n=2$ then we are done by applying Theorem \ref{equal sym}(1) to $k=1$. 
\par Let $n\geq 3$. Then $L_p(z)=\frac{(n-1)(z^n-1)}{nz^n}$ and $L_{p'}(z)=\frac{n-2}{n-1}$. Therefore,  
	$$C_p(z)=\frac{(2n^2-3n+1)z^{2n}+2(2n-1)z^n-(n-1)}{2n^2z^{2n-1}}$$
	and $$C_p'(z)=\frac{(n-1)(2n-1)(z^n-1)^2}{2n^2z^{2n}}.$$
 There are $n$ immediate basins of attraction corresponding to the $n$ roots of $p$. Let $\mathcal{A}_{\omega^j}$ be the immediate basin corresponding to $\omega^j, j =0,1,2,\cdots,n-1$ where  $\omega=e^{\frac{i 2 \pi}{n}}$. 
	\par  
For   $x>1$, $C_p'(x)> 0$ and  $$C_p(x)-x=-\frac{\{(3n-1)x^n-(n-1)\}(x^n-1)}{2n^2x^{2n-1}}=-\frac{\{(n-1)(x^n-1)+2nx^n\}(x^n-1)}{2n^2x^{2n-1}}<0.$$
	Therefore  $\lim\limits_{n \to \infty} C_p^n(x) = 1$ for all  $  x \geq 1$ by Lemma~\ref{monotone}(2). In other words,  $[1, \infty)\subset \mathcal{A}_1$ and  hence $\mathcal{A}_{\omega^0}=\mathcal{A}_1$ is unbounded. As $\Sigma p=\{\lambda z: \lambda^n=1\} \subseteq \Sigma  C_p $ (by Theorem~\ref{symmetry-inclusion}), each $\mathcal{A}_{w^j}$ is unbounded. There is no critical point in $\mathcal{F}(C_p)$ other than the $n$-th roots of unity. It follows from Lemma~\ref{basic-dynamics-lemma} that $\mathcal{F}(C_p)=\bigcup\limits_{j=0}^{n-1}A_j$ where $A_j$ is the basin (not immediate basin) of attraction of $\omega^j$.
	\par 
	Let $U'$ be a Fatou component different from each immediate basin $\mathcal{A}_{\omega^j}$. Then there is a $k$ such that $C_p ^k (U')= \mathcal{A}_{\omega^j}$ for some $j=0,1,2,\cdots, n-1$. By Lemma~\ref{bounded}, $U'$ is bounded. Thus $\mathcal{A}_{\omega^j}$s are the only unbounded Fatou components of $C_p$. Note that $\sigma^j (\mathcal{A}_1)=\mathcal{A}_{\omega^j}$ where $\sigma(z)=\omega z$.
Now, it follows from   Theorem~\ref{no trans} and Lemma~\ref{symmetry-equality} that $\Sigma  C_p =\Sigma p$.
	
	\par 
 The point $0$ is the only pole of $C_p$ and is on the boundary of each $\mathcal{A}_{\omega^j}$  by Lemma~\ref{Nayak_Pal}. The unboundedness of $\mathcal{A}_{\omega^j}$s gives that the Julia component containing $0$ is unbounded. Now it follows from Lemma~\ref{connected J'set} that   $\mathcal{J}(C_p)$ is connected.
\end{proof}
\begin{Remark}
		The solutions of $L_p(z)=-2$ are the extraneous fixed points of $C_p$ and they are precisely the solutions of $z^n=\frac{n-1}{3n-1}$. All these extraneous fixed points have the same multiplier $\frac{2(2n-1)}{n-1}$ which is greater than $1$ and hence these are repelling. 
\end{Remark}
\subsection{Cubic polynomials}
 We present the proof of Theorem \ref{equal sym}(3).
\begin{proof}[Proof of Theorem \ref{equal sym}(3)]
Every cubic centered polynomial is of the form $Az^3+az+b$ for some $A,a,b\in \mathbb{C}$, $A\neq 0$. By Observation \ref{normalization_1}(1), without loss of generality we assume that $p(z)=z^3 +az +b$, where $a,b\in \mathbb{C}$. 
As $p$ is not a monomial and $\Sigma p$ is non-trivial, exactly one of  $a, b$ is zero.
\begin{enumerate}
\item For $a=0$, $p(z)$ is unicritical and is already taken care of in Theorem \ref{equal sym}(2).
\item
Let  $b=0$, i.e., $p(z)=z(z^2+a)$. In view of Remark \ref{normalization}, we assume without loss of any generality that $a=-1$ i.e., $p(z)=z(z^2-1)$. Then  $L_p(z)=\frac{6z^2(z^2-1)}{(3z^2-1)^2}$ and $L_{p'}(z)=\frac{3z^2-1}{6z^2}$. The Chebyshev's method applied to $p$ is
$$	C_p(z)=  
		\frac{z^3(15z^4-6z^2-1)}{(3z^2-1)^3}$$
and 
\begin{equation}\label{deri_d3r2}
C_p'(z)=\frac{3z^2(z^2-1)^2(15z^2+1)}{(3z^2-1)^4}.
\end{equation}

 Let $\mathcal{A}_{-1}$, $\mathcal{A}_0$ and $\mathcal{A}_1$ be the immediate basins of attraction corresponding to the superattracting fixed points $-1$, $0$ and $1$ of $C_p$ respectively.  In addition to $0, \pm 1$, the critical points of $C_p$ are multiple poles $\tau_1=-\frac{1}{\sqrt{3}}$, $\tau_2=\frac{1}{\sqrt{3}}$, each is counted twice as a critical point. The simple critical points are $\pm \frac{i}{\sqrt{15}}$. Since the poles are in the Julia set, in order to  determine the Fatou set of $C_p$ completely, we need to know the forward orbits of $\pm \frac{i}{\sqrt{15}}$. 
	\par We claim that the imaginary axis is invariant and is contained in $\mathcal{A}_0$.
	For $y\in \mathbb{R}$, $C_p(iy)=i\frac{y^3(15y^4+6y^2-1)}{(3y^2+1)^3}=i\varphi(y)$ where $$\varphi(y)=\frac{y^3(15y^4+6y^2-1)}{(3y^2+1)^3}.$$ Then $C_p^n(iy)=i\varphi^n(y)$ for all $y\in \mathbb{R}$ and $n \geq 1$.  Now, the function $\varphi$ has three real roots namely, $0$, $z_1=-\sqrt{\frac{2\sqrt{6}-3}{15}}$ and $z_2=\sqrt{\frac{2\sqrt{6}-3}{15}}$. The nonzero roots are simple, which gives that $\varphi(y)<0$ for $y\in (-\infty,z_1)\cup (0,z_2)$ and   $\varphi(y)> 0$ for $y\in (z_1,0)\cup (z_2,\infty)$. Since $C_p(iy)=i\varphi(y)$, $C_p'(iy)i=i\varphi'(y)$ and it follows from Equation \ref{deri_d3r2} that $\varphi'(y)=\frac{3y^2(y^2+1)^2(15y^2-1)}{(3y^2+1)^4}$.
		The function $\varphi$ is decreasing in $(c_1,c_2)$ as $\varphi'(y)<0$ for $y\in (c_1,c_2)$. Apart from $0$, $\varphi$ has two real critical points, namely $c_1=-\frac{1}{\sqrt{15}}$ and $c_2=\frac{1}{\sqrt{15}}$. As $\frac{2\sqrt{6}-3}{15}>\frac{1}{15}$, $z_2 >c_2$ and therefore $z_1<c_1$. The critical values are $c_1^*=\varphi(c_1)=-\frac{5c_1}{3^5}=\frac{5}{3^5\sqrt{15}}<c_2$ and $c_2^*=\varphi(c_2)=-\frac{5c_2}{3^5\sqrt{15}}=-\frac{5}{3^5\sqrt{15}}>c_1$. In other words, $[c_2^*,c_1^*]\subsetneq [c_1,c_2]$. In Fig. \ref{phi_d3}(b), the  green dots are the roots, red dots are the nonzero critical points and the blue ones are the nonzero critical values of $\varphi$.
	Note that  $|\varphi'(y)|< 3 y^2(y^2+1)^2< 3 (c_1^*)^2 ((c_1^*)^2 +1)^2< \frac{3}{4}$ for all $|y| < c_1^*$. It follows from the Contraction mapping principle that  for all   $y\in [c_2^*,c_1^*]$,   $\lim\limits_{n \to \infty} \varphi^n(y) =0$. Since, $\varphi([z_1,0))=[0,c_1^*]$, we have 
	\begin{equation}
	 \lim\limits_{n \to \infty} \varphi^n(y) =0~\mbox{for all}~  y \in [z_1,0).
	 \label{varphi_1}
	\end{equation}

	\begin{figure}[h!]
		\begin{subfigure}{.5\textwidth}
			\centering
			\includegraphics[width=0.9\linewidth]{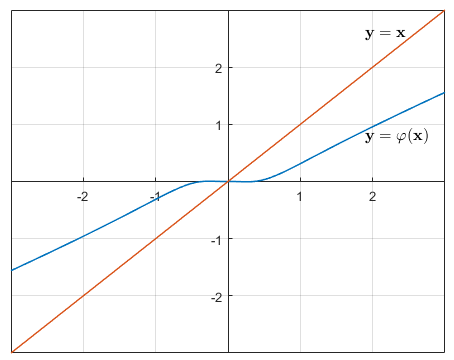}
			\caption{The graph of $\varphi$}
		\end{subfigure}
		\begin{subfigure}{.5\textwidth}
			\centering
			\includegraphics[width=0.9\linewidth]{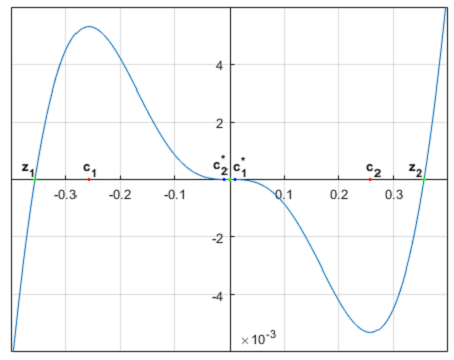}
			\caption{The zoomed image of $\varphi$ near origin }
		\end{subfigure}
		\caption{The real dynamics of $\varphi$ }
		\label{phi_d3}
	\end{figure}
 Note that
 $\varphi(y)-y=-\frac{y(12y^6+21y^4+10y^2+1)}{(3y^2+1)^3}$  and therefore, $\varphi(y)>y$ whenever $y<0$. Also $\varphi$ is increasing in $(-\infty, z_1)$. If for any $y < z_1$, $\varphi^n(y) < z_1$ for all $n$ then $\{\varphi^n(y)\}_{n>0}$ would be an increasing sequence bounded above by $z_1$ and hence must converge. Its limit must be less than or equal to $z_1$ and a fixed point of $\varphi$, which is not possible. Thus, for each $y < z_1$ there is an $n_y$ such that $\varphi^n (y) < z_1$ for all $n < n_y$ and $0> \varphi^{n_y} (y) \geq  z_1$. It now follows from Equation~\ref{varphi_1} that  $\lim\limits_{n \to \infty} \varphi^n(y) =0$ for all $y < z_1$.   As $\varphi(-y)=-\varphi(y)$ and $\varphi^n(-y)=-\varphi^n(y)$ for all  $y\in \mathbb{R}$ and $n \geq 1$, we conclude that,  $\lim\limits_{n \to \infty} \varphi^n(y) =0$ for all $y \in \mathbb{R}$. This gives that the imaginary axis is contained in $\mathcal{A}_0$ and hence $\mathcal{A}_0$ is unbounded.
 
 \par 
  Note that $C_p'(x)>0$ for each $ x< -1$ (see Equation \ref{deri_d3r2}). Also
   $$C_p(x)-x=-\frac{12x(x^2-\frac{9+\sqrt{33}}{24})(x^2-\frac{9-\sqrt{33}}{24})(x^2-1)}{(3x^2-1)^3}.$$ 
  Since  $\frac{9+\sqrt{33}}{24},\frac{9-\sqrt{33}}{24}<1$, $C_p(x)>x$ for all $x <-1$ (see Figure \ref{C_p d3r2}). 
   \begin{figure}[h!]
   	\begin{subfigure}{.5\textwidth}
   	\centering
   	\includegraphics[width=0.93\linewidth]{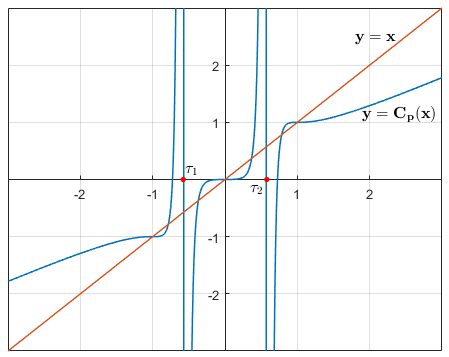}
   	\caption{The graph of $C_p$ }
   	\label{C_p d3r2}
\end{subfigure}
\begin{subfigure}{.5\textwidth}
	   	\centering
 	\includegraphics[width=0.9\linewidth]{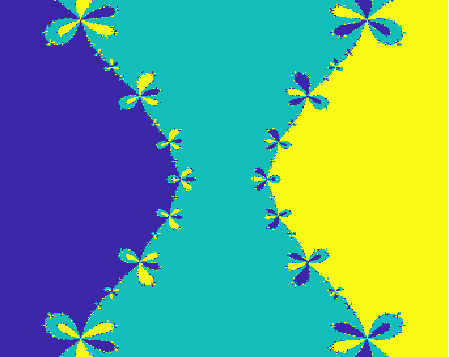}
	\caption{The Julia set of  $C_p$}
	\label{J_set_d3r2}
\end{subfigure}
\caption{ Chebyshev's method $C_p$ for $p(z)=z(z^2-1)$.}
 \end{figure}
Consequently,  $\lim\limits_{n \to \infty} C_p ^n (x)=-1$ for all $x  \leq -1$ by Lemma~\ref{monotone}(1). In other words, $(-\infty, -1] \subset \mathcal{A}_{-1}$. As $\Sigma p=\{I, z\mapsto -z\}\subseteq \Sigma C_p$ (by Theorem \ref{symmetry-inclusion}), $ [1, \infty)\subset \mathcal{A}_1$. Hence  $\mathcal{A}_{-1}$ and $\mathcal{A}_1$ are unbounded. 
\par  
All the critical points (except the poles which are in the Julia set) are in $\mathcal{A}_0 \cup \mathcal{A}_{-1} \cup \mathcal{A}_{1}$. The Fatou set of $C_p$ is the union of the three basins of attraction corresponding to $0,-1$ and $1$ by Lemma~\ref{basic-dynamics-lemma}.  It follows from Lemma \ref{bounded} that every Fatou component different from the three immediate basins are bounded.  In other words, there are exactly three unbounded Fatou components of $C_p$, namely  $\mathcal{A}_0, \mathcal{A}_{-1}$ and $\mathcal{A}_{1}$. Note that $\sigma(\mathcal{A}_{1}) =\mathcal{A}_{-1}$ where $\sigma(z)=-z$.
It follows from   Theorem~\ref{no trans} and  Lemma~\ref{symmetry-equality} that $\Sigma C_p =\Sigma p$.
\par
 By Lemma \ref{Nayak_Pal}, the boundary $\partial \mathcal{A}_0$ of $\mathcal{A}_0$ contains at least one pole. As $\mathcal{A}_0$ is symmetric about the imaginary axis, both the poles $\tau_1$ and $\tau_2$ are in $\partial \mathcal{A}_0$. By the same lemma, each of $\mathcal{A}_{-1}$ and $\mathcal{A}_{1}$ contains a pole in its boundary. Since these are separated by the imaginary axis,  $\tau_1\in \partial\mathcal{A}_{-1}$ and $\tau_2\in\partial\mathcal{A}_1$. Since $\mathcal{A}_{-1} $ and $\mathcal{A}_{0}$ are unbounded, the Julia component containing  $\tau_1$ is unbounded. Similarly, the Julia component containing  $\tau_2$  is unbounded giving that the unbounded Julia component contains both the poles. Thus by Lemma \ref{connected J'set}, $\mathcal{J}(C_p)$ is connected. 
\end{enumerate}
\end{proof}
\begin{Remark}
	\item  For $p(z)=z(z^2-1)$, the finite extraneous fixed points of $C_p$ are the solutions of $L_p(z)=-2$ i.e.,  $12z^4-9z^2+1=0.$ The solutions are $z=\pm \sqrt{\frac{9+\sqrt{33}}{24}}$ and $z=\pm \sqrt{\frac{9-\sqrt{33}}{24}}$. Since the multiplier of an extraneous fixed point $z$ is $\lambda(z)=5+\frac{1}{3z^2}$, $\lambda(\pm \sqrt{\frac{9+\sqrt{33}}{24}})=\frac{13}{2}-\frac{\sqrt{33}}{6}>1$ and $\lambda(\pm \sqrt{\frac{9-\sqrt{33}}{24}})=\frac{13}{2}+\frac{\sqrt{33}}{6}>1$. Hence, all the extraneous fixed points are repelling.
\end{Remark}
\subsection{Quartic polynomials}
Now we prove the final part of Theorem \ref{equal sym}.
\begin{proof}[Proof of Theorem~\ref{equal sym}(4)] Every quartic and centered polynomial is of the form $Az^4+az^2+bz+c$ for some $A,a,b,c\in \mathbb{C}$, $A\neq 0$. In view of Observation~\ref{normalization_1}(1), without loss of generality we assume that $p(z)=z^4+az^2+bz+c$ for some  $a,b,c\in \mathbb{C}$.
	By the hypotheses, $0$ is a root of $p$, i.e., $c=0$. Since $\Sigma p$ is non-trivial, exactly one of  $a$ and $b$ is $0$ giving that  $p(z)=z^2(z^2+a)$ or $p(z)=z(z^3+b)$.
	We provide the proofs for these two cases separately.
	\begin{enumerate}
		\item Let $p(z)=z^2(z^2+a)$. By Remark~\ref{normalization}, we assume $a=-1$, i.e., $p(z)=z^2(z^2-1)$. Then $L_p(z)=\frac{(z^2-1)(6z^2-1)}{2(2z^2-1)^2}$, $L_{p'}(z)=\frac{12z^2(2z^2-1)}{(6z^2-1)^2}$, $$C_p(z)=\frac{z(42z^6-51z^4+20z^2-3)}{8(2z^2-1)^3}$$ and  $$C_p'(z)=\frac{3(z^2-1)^2(28z^4-8z^2+1)}{8(2z^2-1)^4}.$$
In addition  to the superattracting fixed points $\pm 1$ of $C_p$ and the poles $\pm \frac{1}{\sqrt{2}}$ of $C_p$ (each of these is also a critical point of $C_p$ with multiplicity $2$), the other critical points of $C_p$ are the solutions of 
\begin{equation}\label{other cr pt}
28z^4-8z^2+1=0.
\end{equation}
This equation has no real solution because for every solution $z$, $z^2$ is non-real.
If $c$ is such a solution then so also $-c, \bar{c}$ and $-\bar{c}$. The Fatou set  $\mathcal{F}(C_p)$ of $C_p$  contains three immediate basins of attraction $\mathcal{A}_{-1}, \mathcal{A}_1$ and $\mathcal{A}_0$ corresponding to the attracting fixed points $-1,1$ and $0$ respectively. By Lemma~\ref{basic-dynamics-lemma},
  $\mathcal{A}_0$ contains a critical point of $C_p$ and that must be a solution of Equation \ref{other cr pt}. Note that $\mathcal{A}_0$ is invariant under $z\mapsto -\bar{z}$ and  therefore it contains all the four solutions of Equation \ref{other cr pt}. All the critical points except the poles (which are in the Julia set) are in   $\mathcal{A}_0 \cup \mathcal{A}_{-1} \cup \mathcal{A}_1$. By Lemma~\ref{basic-dynamics-lemma}, $\mathcal{F}(C_p)$ is the union of the basins of attraction of $0,-1$ and $1$. 
\par 
We shall show that $\mathcal{A}_0$ is unbounded by establishing  that it contains  the imaginary axis. For $y\in \mathbb{R}$,
$C_p(iy)=i\left[\frac{y(42y^6+51y^4+20y^2+3)}{8(2y^2+1)^3}\right]=i\varphi(y) $
where $$\varphi(y)=\frac{y(42y^6+51y^4+20y^2+3)}{8(2y^2+1)^3}.$$ Thus, $ C_p^n(iy)=i\varphi^n(y) $ for all $n \geq 1$.
Since $\varphi'(y)=\frac{3(y^2+1)^2(28y^4+8y^2+1)}{8(2y^2+1)^4}>0$  and  $$\varphi(y)-y=-\frac{y(22y^6+45y^4+28y^2+5)}{8(2y^2+1)^3}>0,$$ for each $y<0$, $\lim\limits_{n \to \infty} \varphi^n(y)=0$ for all $y \leq 0$ by Lemma~\ref{monotone}(1). As $\varphi(-y)=-\varphi(y)$, we also have $\lim\limits_{n \to \infty} \varphi^n(y)=0$ for all $y>0$. Therefore, the imaginary axis is contained in $\mathcal{A}_0$. 

\begin{figure}[h!]
	\begin{subfigure}{.5\textwidth}
		\centering
		\includegraphics[width=0.9\linewidth]{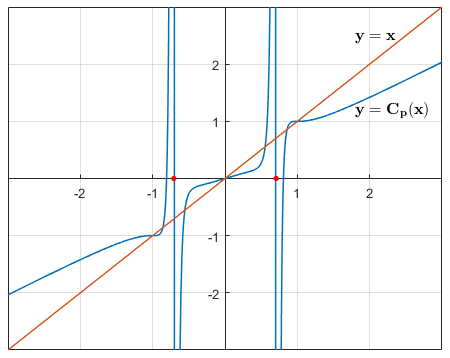}
		\caption{The graph of $C_p$}
	\end{subfigure}
	\begin{subfigure}{.5\textwidth}
		\centering
		\includegraphics[width=0.9\linewidth]{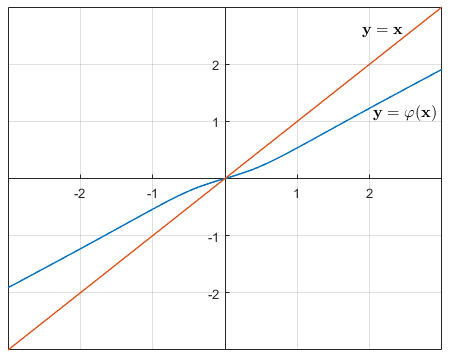}
		\caption{The graph of $\varphi$ near origin }
	\end{subfigure}
	\caption{The real dynamics}
	\label{phi_d4}
\end{figure}
\par 
In order to show that $\mathcal{A}_{-1}$ is unbounded, we now analyze $C_p$ on $(-\infty, -1)$. 
For $y < -1$, $C_p (y)-y=-y\frac{\{(22 y^2 -1)(y^2 -1)+4\} (y^2 -1)}{8 (2 y^2 -1)^3}>0$ and $C_p '(y)>0$ (as $28 y^4-8y^2+1>0$). By Lemma~\ref{monotone}(1), $\lim\limits_{n \to \infty}C_p ^n (y)=-1$ for all $y \leq -1$. In other words, $\mathcal{A}_{-1}$ is unbounded. As $\Sigma p=\{I, z\mapsto -z\}\subseteq \Sigma C_p$ (by Theorem \ref{symmetry-inclusion}), $ (1, \infty)\subset \mathcal{A}_1$. Hence  $\mathcal{A}_{-1}$ and $\mathcal{A}_1$ are unbounded.  By Lemma \ref{bounded}, every Fatou component other than the immediate basins are bounded. In other words, $\mathcal{A}_{0}, \mathcal{A}_{-1}$ and $\mathcal{A}_1$ are the only unbounded Fatou components of $C_p$.  Noting that $\sigma(\mathcal{A}_1)=\mathcal{A}_{-1}$ where $\sigma(z)=-z$, we have $\Sigma C_p = \Sigma p$ by   Theorem~\ref{no trans} and Lemma~\ref{symmetry-equality}.
\par 
That $\mathcal{J}(C_p)$ is connected follows from the same arguments used in the second case of the proof of Theorem~\ref{equal sym}(3).  
\begin{figure}[h!]
	\centering
	\includegraphics[width=0.5\linewidth]{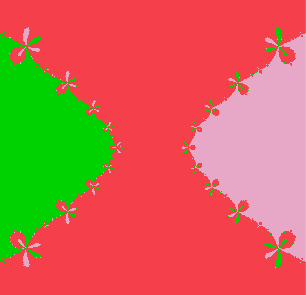}
	\caption{The Julia set of  $C_p$ for $p(z)=z^2(z^2-1)$}
	\label{J_set_d4r2}
\end{figure}
\item Let  $p(z)=z(z^3+b)$. In view of Remark~\ref{normalization}, we assume without loss of generality that $b=-1$ i.e.,  $p(z)=z(z^3-1)$. Then $L_p(z)=\frac{12z^3(z^3-1)}{(4z^3-1)^2}$, $L_{p'}(z)=\frac{4z^3-1}{6z^3}$,  
\begin{align}
&\label{C_p_exp} C_p(z)=\frac{3z^4(14z^6-4z^3-1)}{(4z^3-1)^3}\\ ~\mbox{and}~ 
&\label{deri} C_p'(z)=\frac{12z^3(z^3-1)^2(14z^3+1)}{(4z^3-1)^4}.
\end{align}
The simple roots $0, 1, \omega, \omega^2$ of $p$ and the poles of $C_p$ are critical points of $C_p$. The other critical points are  the solutions of 
\begin{equation}\label{crit}
14z^3+1=0
\end{equation}
and are simple. Note that one of these is a negative real number, say $c=-\left(\frac{1}{14}\right)^{\frac{1}{3}}$.
\par 
Now we look at the real dynamics of $C_p$. The zeros of $C_p$ on the real line are $z_1=-\left(\frac{3\sqrt{2}-2}{14}\right)^{\frac{1}{3}}$, $0$ and $z_2=\left(\frac{2+3\sqrt{2}}{14}\right)^{\frac{1}{3}}$. The extraneous fixed points of $C_p$ are the solutions of $L_p(z)=-2$  which is $22z^6-14z^3+1=0.$
If $z$ is a solution then $z^3=\frac{7+3\sqrt{3}}{22}$ and $z^3=\frac{7-3\sqrt{3}}{22}$. The real extraneous fixed points are 
 $e_1=\left(\frac{7-3\sqrt{3}}{22}\right)^\frac{1}{3}$ and $e_2=\left(\frac{7+3\sqrt{3}}{22}\right)^\frac{1}{3}$. Note that $0<e_1<e_2<1$. The real critical points are $c, 0$, $\tau=\left(\frac{1}{4}\right)^{\frac{1}{3}}$ and $1$. Then $c^*=C_p(c)=\frac{7}{108\sqrt[3]{14}}$. To establish the ordering of these points, we proceed as follows.
\begin{enumerate}
	\item Since $c^3-z_1^3=\frac{3(\sqrt{2}-1)}{14}<0$, $z_1<c$.
	\item $e_1^3-(c^{*})^3=\frac{7\times108^3\times (7-3\sqrt{3})-(7^3\times11)}{11\times 14\times 108^3}>\frac{7\times108^3-(7^3\times11)}{11\times 14\times 108^3}>0$  gives that $c^*<e_1$.
	\item $\tau^3-e_1^3=\frac{3(2\sqrt{3}-1)}{44}>0$ implies that $e_1<\tau$.
	\item As $z_2^3-\tau^3=\frac{3(2\sqrt{2}-1)}{28}>0$, we get $\tau<z_2$.
	\item $e_2^3-z_2^3=\frac{3(9-4\sqrt{2})+21(\sqrt{3}-\sqrt{2})}{11\times 14}>0$  implies that $z_2<e_2$.
\end{enumerate}
The inequations (a-e) are put together as   
\begin{equation}\label{compare}
z_1<c<0<c^*<e_1<\tau<z_2<e_2<1.
\end{equation}

Figure \ref{plot_d4r3} illustrates the graph of $C_p$ on the real line. The blue dots represent the nonzero roots  $z_1$ and $z_2$ of $C_p$, the extraneous fixed points $e_1$ and $e_2$ are shown  as magenta dots whereas the red dots represent the critical points $c,0$ and $1$, and the green dots represent the critical value $c^*$.
 \begin{figure}[h!]
	\begin{subfigure}{.5\textwidth}
		\centering
		\includegraphics[width=0.93\linewidth]{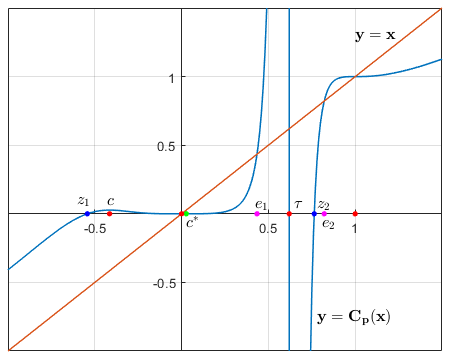}
		\caption{The graph of $C_p$ on the real axis}
		\label{plot_d4r3}
	\end{subfigure}
	\begin{subfigure}{.5\textwidth}
		\centering
		\includegraphics[width=0.9\linewidth]{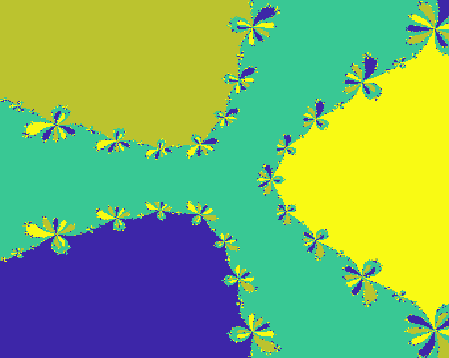}
		\caption{The Julia set of  $C_p$ for $p(z)=z(z^3-1)$}
		\label{J_set_d4r3}
	\end{subfigure}
	\caption{Chebyshev's method $C_p$ for $p(z)=z(z^3-1)$}
\end{figure}
\par 
 Note that
\begin{equation}\label{comp_x}
C_p(x)-x=-\frac{22x(x^3-1)(x^3-\frac{7-3\sqrt{3}}{22})(x^3-\frac{7+3\sqrt{3}}{22})}{(4x^3-1)^3}.
\end{equation}
It follows from Equation \ref{comp_x} that $C_p(x) > x$ for all $x <z_1$. If $C_p ^n (x) < z_1$ for all $n$ then $\{C_p ^n (x)\}_{n>0}$ would converge and the limit point must be a fixed point of $C_p$ and less than or equal to $z_1$. However this is not possible. Therefore, for each $x\in (-\infty, z_1)$, there exists a natural number $n_x$ such that $C_p ^n (x) < z_1$ for all $n < n_x$ and $C_p^{n_x}(x)\in [z_1,0)$. Since $C_p$ is strictly increasing in $(-\infty, c)$ and decreasing in $(c,0)$, $C_p([z_1,0))=[0,c^*]$.  As $C_p'(x)>0$ for all $x>0$, it is so in $[0,c^*]$. Further, $C_p(x)<x$ whenever $x\in (0,c^*]$ by Relation~\ref{compare}. Both these facts along with $C_p (0)=0$ imply that $\{C_p ^n (x)\}_{n>0}$ is a decreasing sequence which is bounded below by $0$. The Monotone convergence theorem gives  that $\lim\limits_{n \to \infty}C_p^n (x)= 0$ for all $x\in [0,c^*]$. Hence  $(-\infty, 0] \subset \mathcal{A}_0$. In particular, $c\in \mathcal{A}_0$ and $\mathcal{A}_0$ is unbounded.

 As $\Sigma p =\{z\mapsto \lambda z: \lambda^3=1\}\subseteq \Sigma C_p$, and for every $\sigma\in \Sigma p$, $\sigma(\mathcal{A}_0)=\mathcal{A}_0$  by Lemma \ref{imm_basin_preserved}. the other two critical points which are the solutions of the Equation \ref{crit} are also in $\mathcal{A}_0$. Therefore, all the critical points except the poles are in  $\mathcal{A}_0 \cup \mathcal{A}_1 \cup \mathcal{A}_{\omega} \cup \mathcal{A}_{\omega^2}$. By Lemma~\ref{basic-dynamics-lemma}, $\mathcal{F}(C_p)$ is the union of the basins  (not only immediate basins) of attraction of the four superattracting fixed points $0,1,\omega$ and $\omega^2$.
  \par 
Note that $C_p '(x)>0$ for each $x >1$ and from Equation \ref{comp_x} we get $C_p(x)<x$ for each $x > 1$. By Lemma~\ref{monotone}(2), $\lim\limits_{n \to \infty} C_p ^n (x)=1$ for all $x  \geq 1$. Thus, $[1, \infty) \subset \mathcal{A}_1$ and $\mathcal{A}_1$ is unbounded. Again, as $\sigma_1(\mathcal{A}_1)=\mathcal{A}_\omega$ and $\sigma_2(\mathcal{A}_1)=\mathcal{A}_{\omega^2}$ for $\sigma_1(z)=e^{\frac{2\pi i}{3}}z$ and $\sigma_2(z)=e^{\frac{4\pi i}{3}}z$, $\mathcal{A}_\omega$ and $\mathcal{A}_{\omega^2}$ are also unbounded.  All the Fatou components different from the immediate basins of the fixed points $0, 1, \omega$ and $\omega^2$ are bounded by Lemma~\ref{bounded}. In other words, the only unbounded Fatou components of $C_p$ are $\mathcal{A}_0, \mathcal{A}_1,  \mathcal{A}_{\omega}$ and $ \mathcal{A}_{\omega^2}$. Note that $\sigma^j(\mathcal{A}_1)=\mathcal{A}_{\omega^j}$ for $j=1,2$ where $\sigma(z)=\omega z$. By   Theorem~\ref{no trans} and  Lemma~\ref{symmetry-equality}, $ \Sigma C_p = \Sigma p$.
 
  The boundary of $\mathcal{A}_0$ contains a pole by Lemma~\ref{Nayak_Pal}. As the set of all poles of $C_p$ as well as $\mathcal{A}_0$ are invariant under every $\sigma \in \Sigma C_p$, all poles of $C_p$ are on $\partial \mathcal{A}_0$. Since  $(-\infty, 0] \subset \mathcal{A}_{0}$, it also follows that the  two rays emanating from the origin with arguments  $\frac{\pi}{3}$ and   $\frac{5\pi}{3}$ are contained in $\mathcal{A}_0$. These rays (and hence $\mathcal{A}_0$) separate the immediate basins  $\mathcal{A}_1, \mathcal{A}_\omega$ and $\mathcal{A}_{\omega^2}$ from each other.

  By Lemma~\ref{Nayak_Pal}, each of the immediate basins $\mathcal{A}_1, \mathcal{A}_\omega$ and $\mathcal{A}_{\omega^2}$ contains a pole in its boundary. These poles are distinct. In fact,  $\tau\in \partial\mathcal{A}_1$, $\tau e^{\frac{2\pi i}{3}}\in \partial\mathcal{A}_\omega$ and $\tau e^{\frac{4\pi i}{3}}\in \partial\mathcal{A}_{\omega^2}$.
Observe that each pole is in the boundary of two unbounded Fatou components. This means that the unbounded Julia component contains all the poles. By Lemma~\ref{connected J'set}, the Julia set of $\mathcal{J}(C_p)$ is connected.
 
	\end{enumerate}
\end{proof}
\begin{Remark} All the extraneous fixed points of $C_p$ are repelling for $p$ satisfying the hypothese of Theorem~\ref{equal sym}(4).
	\begin{enumerate}
		\item 
 For $p(z)=z^2 (z^2-1)$, the extraneous fixed points of $C_p$ are the solutions of $L_p(z)=-2$ which are nothing but the solutions of $22z^4-23z^2+5=0$. The solutions  are $z=\pm \sqrt{\frac{23+\sqrt{89}}{44}}$ and $z=\pm \sqrt{\frac{23-\sqrt{89}}{44}}$. As $ \frac{23+\sqrt{89}}{44},\frac{23-\sqrt{89}}{44}<1 $, all the extraneous fixed points are in $(-1,1)$. The multiplier $\lambda(z)$ of an extraneous fixed point $z$ is given by the formula $\lambda(z)=2(3-L_{p'}(z))=6\left[1-\frac{4z^2(2z^2-1)}{(6z^2-1)^2}\right]=6-\alpha(z^2)$, where $\alpha(z)=\frac{24z(2z-1)}{(6z-1)^2}$. Now $\alpha(\frac{23+\sqrt{89}}{44})=\frac{7623\sqrt{89}-49731}{30976}<1$ and  $\lambda(z=\pm \sqrt{\frac{23+\sqrt{89}}{44}})=6-\alpha(\frac{23+\sqrt{89}}{44})>1$. Similarly $\alpha(\frac{23-\sqrt{89}}{44})=-\frac{7623\sqrt{89}+49731}{30976}$ gives that  $\lambda(z=\pm \sqrt{\frac{23-\sqrt{89}}{44}})>1$.
	\item For $p(z)=z(z^3 -1)$, the extraneous fixed points, we need to solve $L_p(z)=-2$ which gives the equation
	$22z^6-14z^3+1=0.$
	The solutions are $z^3=\frac{7+3\sqrt{3}}{22}$ and $z^3=\frac{7-3\sqrt{3}}{22}$. The multiplier of an extraneous fixed point $z$ is given by $\lambda(z)=\frac{14z^3+1}{3z^3}$. Those satisfying $z^3=\frac{7+3\sqrt{3}}{22}$ has multiplier $7-\sqrt{3}>1$, whereas the multiplier of those extraneous fixed points satisfying  $z^3=\frac{7-3\sqrt{3}}{22}$ is $7+\sqrt{3}>1$.

	\end{enumerate}
\end{Remark}
\section{Concluding remarks }
We conclude with following remarks.
\begin{enumerate}
\item Recall from Remark \ref{equal_condition} that the Theorem \ref{symmetry-inclusion} is true for every polynomial $p$ with centroid $\xi\neq 0$, whenever the symmetry groups of $p$ and $p\circ T$ are isomorphic, where $T(z)=z+\xi$. Since we are concerned with those $p$ for which $p\circ T$ is not a monomial, $\Sigma p\circ T$ is finite. This gives that $\Sigma p$ is isomorphic to $\Sigma p\circ T$ if and only if they have the same order. We now provide a sufficient criterion for which $\Sigma p$ and $\Sigma p\circ T$ are isomorphic.
\begin{lemma}\label{iso_sym}
	If $p$ is a polynomial with centroid $\xi$, $p(\xi)\neq 0$, and $p(\xi)\neq \xi$ then $\Sigma p$ and $\Sigma p\circ T$ are isomorphic, where $T(z)=z+\xi$.
\end{lemma}
\begin{proof}
	Without loss of generality we consider $p$ is monic. Let $p(z)=z^d+a_{d-1}z^{d-1}+\dots +a_1z+a_0$. Then $\xi=-\frac{a_{d-1}}{d}$. Considering $g(z)=p\circ T(z)$, we have
	\begin{align*}
	g(z)=&(z+\xi)^d+a_{d-1}(z+\xi)^{d-1}+\dots+a_0\\
	=&z^d+d\xi z^{d-1}+\dots+\xi^d\\
	+&a_{d-1}z^{d-1}+\dots +a_{d-1}\xi^{d-1}\\
	\vdots\\
	+&a_1\xi+a_0
	\end{align*}
	Thus, $g(z)=z^d+b_{d-2}z^{d-2}+\dots+p(\xi)$ for some $b_{d-j}\in \mathbb{C}$ where $j=2,3,\dots, d-1$.\\
	Now
	\begin{align*}
	f(z)&=T^{-1}(p(T(z)))=g(z)-\xi\\
	&=z^d+b_{d-2}z^{d-2}+\dots+p(\xi)-\xi.
	\end{align*} 
	Note that the constant terms of $g$ and $f$ are non-zero by the assumptions. As both $f$ and $g$ are normalized polynomial differing by the constant terms only, $\Sigma f=\Sigma g$. By Lemma \ref{Cg}, $\Sigma f=T^{-1}(\Sigma p)T$. Hence the orders of $\Sigma p$ and $\Sigma g$ are equal.
\end{proof}
If a unicritical polynomial is of the form $p(z)=\lambda(z-a)^n+c$ where $a,c$ are non-zero with $a\neq c$, then its centroid $a$ is neither a root nor a fixed point of it. It follows from Lemma \ref{iso_sym} that $\Sigma p$ and $\Sigma g$ are of the same order, where $g(z)=p(z+a)$. Since $g$ is unicritical and centered, we have $\Sigma g= \Sigma C_g$ by Theorem \ref{equal sym}(2). Now $C_p$ is conjugate to $C_g$ since the Chebyshev's method satisfies the Scaling theorem. Therefore, we have the following.
\begin{theorem}
	If $p$ is a unicritical polynomial whose centroid is neither a root nor a fixed point of it then $\Sigma p=\Sigma C_p$.
\end{theorem}
This is a refinement of Theorem \ref{equal sym}(2).
\par 
Note that the conditions in Lemma \ref{iso_sym} are not necessary. For example, consider $p(z)=z^4-4z^3+7z^2-6z+3$ whose centroid $\xi=1$. Then $g(z)=p(z+\xi)=z^4+z^2+1$ and $f(z)=g(z)-\xi=z^2(z^2+1)$. Therefore $\Sigma g=\Sigma f$ and thus $o(\Sigma p)=o(\Sigma g)$. However, $p(1)=1$.
\item For a rational map $R$, let $\mathcal{P}_R=\bigcup\limits_{c} \{R^n(c): n \geq 0
~\mbox{and}~ c\text{ is a critical point of } R\}$. If $\overline{\mathcal{P}_R}\cap \mathcal{J}(R)$ is finite then $R$ is said to be a \textit{geometrically finite map}.  It follows from Lemma~\ref{basic-dynamics-lemma} that geometrically finite maps donot contain any Siegel disk  or Herman ring. 
In all the cases considered in Theorem~\ref{equal sym},  all the critical points, except the poles are in the basins of attraction of the (super)attracting fixed points of $C_p$ and the poles are in the Julia set of $C_p$. Hence $C_p$ is a geometrically finite map. In particular, for cubic $p$ it is shown that if $\Sigma p$ is non-trivial then $C_{p}$ is geometrically finite. In this case, all the extraneous fixed points are found to be repelling. In \cite{Nayak-Pal2022}, it is found that $C_{p_\lambda}$ is geometrically finite for $p_{\lambda}(z)=z^3+3z+\frac{3\lambda^2-39\lambda+124}{(5-\lambda)\sqrt{5-\lambda}}$ whenever $\lambda\in [-1,1]$ and in this case an extraneous fixed point is either attracting or parabolic. The number $\lambda$ indeed represents the multiplier of an extraneous fixed point of $C_{p_\lambda}$. It can be suitably choosen so that $C_{p_\lambda}$ has an invariant Siegel disk and hence $C_p$ is not geometrically finite. With all these observations, a complete characterization of all the cubic polynomials whose Chebyshev's method is   geometrically finite becomes a  relevant question.  
\par In all the cases stated in the previous paragraph, the Julia set of the Chebyshev's method is found to be connected. In \cite{Lei96} Lei and Yongcheng prove that the Julia set of a geometrically finite rational map is locally connected whenever it is connected. Thus  the Julia set is locally connected in all the aforementioned situations. It also  seems important to determine all the cubic polynomials $p$ for which $\mathcal{J}(C_p)$ is connected irrespective of whether $C_p$ is geometrically finite or not. The same issue for polynomials of higher degree also remains to be investigated.

\item That $\Sigma p$ is non-trivial is crucial in the proof of each case of Theorem~\ref{equal sym}. Though $\Sigma C_p$ is completely understood for all cubic polynomials with non-trivial $\Sigma p$, the case of quartic polynomials remains incomplete. In particular, polynomials of the form $p(z)=z^4+bz^2+c $ where $ b, c \neq 0$   is not studied in this paper. It can be seen that this is a one-parameter family namely $p(z)=z^4-(a+1)z^2+a$, $a\in \mathbb{C}\setminus \{0,1\}$.  A similar study of $\Sigma C_p$ seems possible using the tools developed in this article.
\par 
Understanding $\Sigma C_p$ is also an interesting problem when $p$ is cubic or quartic and  $\Sigma p$ is trivial.
\end{enumerate}
\section*{Acknowledgement}
The second author is supported by the University Grants Commission, Govt. of India.

\end{document}